\documentclass[a4paper,12pt]{amsart}

\usepackage{bbm}
\usepackage{units}
\usepackage{mathtools}
\usepackage{cancel}
\usepackage{amsmath}
\usepackage{amsfonts}
\usepackage{latexsym}
\usepackage{amssymb}
\usepackage{mathrsfs}
\usepackage{amscd}
\usepackage{hyperref}
\usepackage{soul}
\usepackage{psfrag}
\usepackage{graphicx}
\usepackage{ulem}
\usepackage{fullpage}
\usepackage[all]{xy}
\usepackage{rotating}
\usepackage{url}
\usepackage{color}
\usepackage{enumerate}
\usepackage{cleveref}

\hypersetup{
	colorlinks=true, 
	linkcolor=red, 
	citecolor=blue,
}

\usepackage{dsfont}
\usepackage{tikz}
\usetikzlibrary{arrows,backgrounds,arrows.meta,decorations.markings}
\usepackage{tikz-cd}

\numberwithin{equation}{section}
\numberwithin{figure}{section}

\theoremstyle{plain}
\newtheorem{thm}{Theorem}[section]

\theoremstyle{plain}
\newtheorem{lem}[thm]{Lemma}

\theoremstyle{remark}
\newtheorem{rem}[thm]{Remark}

\theoremstyle{plain}
\newtheorem{cor}[thm]{Corollary}

\theoremstyle{definition}
\newtheorem{defn}[thm]{Definition}

\theoremstyle{definition}

\theoremstyle{definition}

\theoremstyle{plain}
\newtheorem{prop}[thm]{Proposition}

\theoremstyle{plain}

\theoremstyle{definition}

\theoremstyle{plain}

\newcommand{\comments}[1]{}
\newcommand{\ra}{\rightarrow}
\newcommand{\rab}{\rangle}

\newcommand{\lab}{\langle}

\newcommand{\mcal}{\mathcal}

\newcommand{\N}{\mathbb N}

\newcommand{\C}{\mathbb{C}}
\newcommand{\R}{\mathbb{R}}

\newcommand{\vlon}{\varepsilon}

\title{Positivstellensatz for C*-tensor categories}
\author{Kajal Das, Mainak Ghosh and Shamindra Ghosh}
\newcommand{\Contact}{{
		\bigskip
		\footnotesize
				
		Mainak Ghosh, \textsc{Beijing Institute of Mathematical Sciences and Applications}\par\nopagebreak
		\textit{E-mail address}: \texttt{main\_ghosh@rediffmail.com} , \texttt{mainakghosh@bimsa.cn}
		
		\medskip
		
		Shamindra Ghosh, \textsc{Stat-Math Unit, Indian Statistical Institute}\par\nopagebreak
		\textit{E-mail address}: \texttt{shamindra.isi@gmail.com}
		
		\medskip
		
		Kajal Das, \textsc{Stat-Math Unit, Indian Statistical Institute}\par\nopagebreak
		\textit{E-mail address}: \texttt{kdas.math@gmail.com}
		
	}}
\begin{document}
	\maketitle
	\global\long\def\vlon{\varepsilon}
	\global\long\def\bt{\bowtie}
	\global\long\def\ul#1{\underline{#1}}
	\global\long\def\ol#1{\overline{#1}}
	\global\long\def\norm#1{\left\|{#1}\right\|}
	\global\long\def\os#1#2{\overset{#1}{#2}}
	\global\long\def\us#1#2{\underset{#1}{#2}}
	\global\long\def\ous#1#2#3{\overset{#1}{\underset{#3}{#2}}}
	\global\long\def\t#1{\text{#1}}
	\global\long\def\lrsuf#1#2#3{\vphantom{#2}_{#1}^{\vphantom{#3}}#2^{#3}}
	\global\long\def\tr{\triangleright}
	\global\long\def\tl{\triangleleft}
	\global\long\def\cc90#1{\begin{sideways}#1\end{sideways}}
	\global\long\def\turnne#1{\begin{turn}{45}{#1}\end{turn}}
	\global\long\def\turnnw#1{\begin{turn}{135}{#1}\end{turn}}
	\global\long\def\turnse#1{\begin{turn}{-45}{#1}\end{turn}}
	\global\long\def\turnsw#1{\begin{turn}{-135}{#1}\end{turn}}
	\global\long\def\fusion#1#2#3{#1 \os{\textstyle{#2}}{\otimes} #3}
	
	\global\long\def\abs#1{\left|{#1}\right	|}
	\global\long\def\red#1{\textcolor{red}{#1}}

\begin{abstract}
 We explore semi-pre-C*-algebras in the context of rigid semisimple C*-tensor categories and using techniques from annular representations, we extend Ozawa's criterion for property (T) in groups to this context. \comments{\sout{for rigid C*-tensor categories. We use techniques from the theory of semi-pre-C*-algebras and annular representations of rigid C*-tensor categories to prove our results}}
\end{abstract}

\section{Introduction}

Rigid C*-tensor categories, introduced in their modern form by R. Longo and J. Roberts \cite{LR}, are structures which can provide a unifying framework for symmetries appearing in a variety of contexts such as representation category of compact quantum groups \cite{NT}, DHR super-selection sectors in algebraic quantum field theory \cite{H96}, bifinite bimoudules over $ II_1 $-factors \cite{J83}, etc. In the last three decades, there has been a great deal of interest in approximation and rigidity properties for rigid, semisimple C*-tensor categories and subfactors.

\vspace*{2mm}
 
S. Popa introduced the concepts of approximation and rigidity properties for subfactors \cite{P94a,P94b,P96,P99}. Subsequently, approximation and rigidity were translated from Popa's original definitions into the categorical setting by Popa and Vaes in \cite{PV}; they introduce a representation theory for standard invariants and rigid C*-tensor categories generalizing the unitary representation theory of groups, encoding the analytical properties in a natural way.

\vspace*{2mm}

Property (T) is one of the two most important concepts in analytic group theory, the other being amenability. Suppose $\Gamma$ is a finitely generated group and $\mu$ be a probability measure on $\Gamma$ whose support is finite and symmetric, and generates $\Gamma$.
Then, the corresponding Laplacian element $\Delta$ is the positive element in the real group algebra $\R[\Gamma]$ defined by 
\[ 
\Delta = \frac{1}{2} \underset{x \in \Gamma}{\sum} \mu(x)\left(1-x \right)^* \left(1-x \right) = 1 - \underset{x \in \Gamma}{\sum} \mu(x) x .
\] 
It is well-known (see \cite{BHV08}) that $\Gamma$ has property (T) if and only if $\Delta$ has a spectral gap, that is, there exists a scalar $ k > 0 $ such that for every orthogonal $\Gamma$-representation $\pi$ on a real Hilbert space $H$, the spectrum of $\pi(\Delta) \in \mcal L(H)$ is contained in $\left\{ 0 \right\} \cup \left[k, \infty \right)$. 
It turns out $ \Delta^2 - k \Delta \geq 0 $ in the full group C*-algebra $C^* (\Gamma)$. In \cite{Oz16}, N. Ozawa proved that $\Delta^2 - k \Delta \geq 0 $ is witnessed in the algebra $\R [\Gamma]$.

\vspace*{2mm}

Our main goal in this paper is to extend Ozawa's criterion for property (T) in groups to the context of rigid, semisimple C*-tensor categories with countably many irreducibles and simple tensor unit. Our approach has been based on techniques from noncommutative real algebraic geometry (as in \cite{S08,Oz13}). One of the most important component in the realm of noncommutative algebraic geometry are semi-pre-C*-algebras (a kind of generalization of unital C*-algebras). In  \cite{Oz13}, N. Ozawa formulated \textit{Connes' embedding conjecture} using the theory of representations of semi-pre-C*-algebras. The theory of semi-pre-C*-algebras in the context of group algebras has been studied in \cite{NeTh13}.

\vspace*{2mm}

We study the fusion algebra $\C [\mcal C]$ of a rigid, semisimple C*-tensor category $ \mcal C $ under the lens of semi-pre-C*-algebras. We are unable to find any such treatment in the literature before.
The first obvious question is whether the fusion algebra becomes a semi-pre-C*-algebra with respect to the canonical cone (that is, finite sum of elements of the form $x^* x$ for $x \in \C[\mcal C]$) .
It appears that this does not hold all the time.
This led us to the formulation of \textit{tube cone} (denoted by $\Sigma^2 \mcal C$) using Ocneanu's tube algebra and the annular representations of $ \mcal C $ \cite{O94,GhJ}.
When one is working with \textit{$*$-positive cones}, it is natural to ask \cite{NeTh13} the following questions:
\begin{itemize}\label{order unit}
	\item [(1)] Is the tube cone \textit{salient} ?
	\item [(2)] Is the unit of the fusion algebra an \textit{order unit} with respect to the tube cone ?  
\end{itemize}
We answer these questions in the affirmative and thereafter, have the first main result.
\begin{thm}
	The fusion algebra is a semi-pre-C*-algebra with respect to the tube cone
\end{thm}

We then turn our attention to prove a version of Ozawa's criterion for property (T) of a C*-tensor category. 
The Laplacian element in the group case was trivially positive because the basis elements of the group algebra are invertible. However, in $ \mcal C $, the simple objects could very well be non-invertible and thereby positivity of the corresponding Laplacian element $ \Delta $ is not apparent.
We need to use that the unit of the fusion algebra is an order unit with respect to the tube cone to prove the following result.
\begin{prop}\label{tubeconepos}
The Laplacian element is positive in the fusion algebra with respect to the tube cone.
\end{prop}

In \cite{NY16}, it has been shown that admissible representations can be understood as objects in the Drinfeld center of the ind-category. In \cite{GhJ}, the authors show that admissible representations have a natural interpretation in the annular representation theory of Jones' planar algebras \cite{J01,JR06}, and the representation theory of the tube algebra of a category, introduced by Ocneanu \cite{O94}. The fusion algebra is a corner of the tube algebra, and in \cite{GhJ} it is shown that admissible representations of the fusion algebra are precisely representations which are restrictions of representations of the whole tube algebra. Since the tube algebra is computable in principle, this provides a method for determining admissible representations.
Using \Cref{tubeconepos} and the fact that weight $\mathbbm{1}$-admissible representations are equivalent to $\Sigma^2 \mcal C$-\textit{positive representations}, we prove the main theorem of the paper.
\begin{thm}
A rigid, semisimple C*-tensor category with simple unit object and tensor-generated by finitely many simple objects, has property (T) if and only if there exists $ k > 0 $ such that $\Delta^2 -k \Delta + \epsilon \mathbbm{1} \in \Sigma^2 \mcal C$ for every $ \epsilon > 0 $.
\end{thm}
There are at least two questions which arise from our work.
\begin{itemize}
	\item [(i)] Can one formulate a notion of residual finite dimensionality (RFD) in the context of rigid C*-tensor categories using the notion of tube cones ? The motivation for this question comes from \cite[Corollary 2]{Oz13}
	
	\item [(ii)] To study actions of RFD C*-tensor categories on unital C*-algebras (cf. \cite{CPJ})
\end{itemize}
We will address these questions in our forthcoming articles.

\vspace*{2mm}

\subsection*{Acknowledgement}
The second author is supported by BJNSF (International Scientists Project, Grant no. 1S24063).

\section{Preliminaries}\label{prelim}
In this section, we lay out the basic results on rigid C*-tensor categories and their tube algebras pertaining to our results. We refer the reader to \cite{NT} for a comprehensive study of rigid C*-tensor categories, and to \cite{GhJ} for a detailed study of tube algebras. Throughout our discussion, we assume the C*-tensor categories to be strict, essentially small and closed under finite direct sums and passage to sub-objects. For a C*-tensor category $\mcal C$, and for two objects $x, y \in \t{Ob}(\mcal C)$, we denote the space of morphisms from $x$ to $y$ by $\mcal C(x,y)$.

\subsection{Rigid C*-tensor categories}\

A rigid C*-tensor category is a C*-tensor category $\mcal C$, in which every object $X$ has a conjugate object $\ol X$ along with morphisms $R \in \mcal C \left(\mathbbm{1}, \ol X \otimes X\right)$  and $\ol R \in \mcal C \left(\mathbbm{1}, X \otimes \ol X\right)$ (where $\mathbbm{1}$ is the unit object of $\mcal C$) satisfying the following conjugate equations :
\[\left(1_{\ol X} \otimes \ol{R}^* \right) \left(R \otimes 1_{\ol X}\right) = 1_{\ol X} \ \ \ \ \t{and} \ \ \ \ \left(1_{X} \otimes R^* \right) \left(\ol R \otimes 1_{X}\right) = 1_{X} . \]
The pair $\left(R,\ol R\right)$ is called a solution to the \textit{conjugate equations} for the pair $\left(X, \ol X \right)$. 
\begin{rem}
	In a rigid C*-tensor category $\mcal C$, conjugate of an object is unique up to isomorphism. More precisely, if $\left(R,\ol R\right)$ is a solution to conjugate equations for $\left(X,\ol X\right)$ and if $\left(R',\ol R'\right)$ for $\left(X,\ol X'\right)$ then $T = \left(1_{\ol X} \otimes \ol R'^* \right) \left(R \otimes 1_{\ol X'}\right) \in \mcal C \left(\ol X', \ol X\right)$ is invertible with inverse $S = \left(1_{\ol X'} \otimes \ol R^* \right) \left(R' \otimes 1_{\ol X}\right) \in \mcal C \left(\ol X, \ol X' \right) $, and $R' = \left(T^{-1} \otimes 1 \right)R$ and $\ol R' = \left(1 \otimes T^*\right)\ol R$ (see \cite[Proposition 2.2.5]{NT}).
\end{rem}

\begin{rem}
	The rigidity condition on a C*-tensor category enforces the morphism spaces to be finite-dimensional (see \cite{NT}).
\end{rem}

\begin{defn}
	A solution $\left(R,\ol R\right)$ to the conjugate equations for the pair $\left(X, \ol X \right)$ is called \textit{standard} if it satisfies the following :
 \begin{equation}\label{stdeqn}
 	R^* (1 \otimes f) R = \ol R^* (f \otimes 1) \ol R  \ \ \ \ \t{for all} \ \ f \in \t{End}(X) .
 \end{equation}
\end{defn}
Conjugate objects and standard solutions are unique up to unitary conjugacy and satisfy various naturality properties with respect to direct sums and tensor products (see \cite{NT}). Through out our discussion, we always consider fixed standard solutions for all irreducibles, and extend those to arbitrary objects by naturality.

The standard solution $\left(R, \ol R\right)$, allows one to define a function $d : \t{Ob}(\mcal C) \to \left(0,\infty\right)$ as
\[d(X) = R^* R = \ol R^* \, \ol R . \]
We call $d(X)$ to be the \textit{categorical dimension} of $X$.

Given a standard solution $\left(R, \ol R\right)$ for $\left(X, \ol X \right)$, \Cref{stdeqn} yields a faithful, positive, tracial functional on $\t{End}(X)$, given by 
 \[\t{Tr}(f) = R^* (1 \otimes f) R = \ol R^* (f \otimes 1) \ol R  \ \ \ \ \t{for all} \ \ f \in \t{End}(X) .\]
 We call the functional $\t{Tr}$ to be the \textit{categorical trace} on $\t{End}(X)$ and is independent of the choice of a standard solution.
 
 Given an object $X \in \t{Ob}(\mcal C)$, we denote by $[X]$ to be isomorphism class of $X$ in $\mcal C$. Let $\textbf{Irr}(\mcal C)$ denote the set of isomorphism classes of simple objects. We will always work with some fixed choice of representatives for all $s \in \textbf{Irr}(\mcal C)$, and we do not distinguish between simple objects and their isomorphism classes.
 
 \vspace*{2mm} 
 
 The \textit{fusion algebra} $\C[\mcal C]$, is defined as the complex linear span of isomorphism classes of simple objects, with
 multiplication given by linear extension of the fusion rules. That means, for simple objects $X$ and $Y$, $[X] \cdot [Y] = \us{Z \in \textbf{Irr}(\mcal C)} \sum [X \otimes Y : Z] [Z]$, where $[X \otimes Y : Z]$ is the vector space dimension $\dim \mcal C (Z, X \otimes Y)$ . This algebra has a $ * $-involution defined by $[X]^* = [\ol X]$ and extended conjugate linearly. Clearly, $[\mathbbm{1}]$ is a unit of $\C[\mcal C]$. This unital $ * $-algebra is a central object of our study in this paper.

 \subsection{Tube algebras}\
 
 \comments{\red{To write few lines about tube algebras}}
 The tube algebra of a rigid C*-tensor category $\mcal C$ was introduced by A. Ocneanu \cite{O94}, in the context of subfactors. There is an one-to-one correspondence between finite-dimensional irreducible representations of the tube algebra and simple objects of the Drinfeld center $Z(\mcal C)$, and hence this has been useful for computing $Z(\mcal C)$ \cite{I}. In general, arbitrary representations of the tube algebra are in one-to-one correspondence with $Z(\t{ind}-\mcal C)$ (see \cite{NY16}).
 
 \vspace*{2mm}

 Suppose $\mcal C$ is a rigid C*-tensor category, the tube algebra $\mcal A$ of $\mcal C$ is defined by the algebraic direct sum
 \[ \mcal A = \us{i,j,k \in \textbf{Irr}(\mcal C)}{\bigoplus} \mcal C(k \otimes i, j \otimes k) . \]
 An element $x \in \mcal A$ is given by a sequence $x^k_{i,j} \in \mcal C(k \otimes i, j\otimes k)$ with only finitely many terms non-zero
 
 The $ * $-algebra structure on $\mcal A$ is defined as follows :
 \[ \left(x \cdot y \right)^k_{i,j} = \us{s,m,l \in \textbf{Irr}(\mcal C)}{\sum} \us{V \in \t{onb}(k, m \otimes l)}{\sum} (1_j \otimes V^*)(x^m_{s,j} \otimes 1_l) (1_m \otimes y^l_{i,s})(V \otimes 1_i) \]
 
 \[\left(x^*\right)^k_{i,j} = \left(\ol R^* \otimes 1_j \otimes 1_k \right)\left(1_k \otimes \left(x^{\ol k}_{j,i}\right)^* \otimes 1_k \right) \left(1_k \otimes 1_i \otimes R \right)\]
 where $x,y \in \mcal A$ and $\left(R, \ol R \right)$ are standard solutions to conjugate equations for $\left( k, \ol k \right)$.
 
 \begin{rem}
 	The tube algebra $\mcal A$ of a rigid C*-tensor category is, in general, not unital. It is unital if and only if the number of isomorphism classes of simple objects in $\mcal C$ is finite.
 \end{rem}
 
 \begin{defn}\cite{GhJ}
 	A \textit{non-degenerate right representation} of the tube algebra $\mcal A$ is a $ * $-homomorphism $\pi : \mcal A \to B\left(H\right)$ for some Hilbert space $H$ with the property that, for $\xi \in H$, $\xi \cdot \pi \left(\mcal A\right) = 0$ implies $\xi = 0$.
 \end{defn}
Following the notations in \cite{VV19}, we will call a Hilbert space $H$ to be a non-degenerate right Hilbert $\mcal A$-module if such a non-degenerate right representation of $\mcal A$ exists on $H$. 
There are many equivalent definitions of property (T) \cite{P99,PV,NY16,GhJ}. However, the one that we will state below, will be the most suitable one for our .

\begin{defn}\cite{NY16}
	Suppose $\mcal C$ is a rigid C*-tensor category with tube algebra $\mcal A$. Consider a non-degenerate right Hilbert $\mcal A$-module $H$. A vector $\xi \in H$ is \textit{invariant} if $\xi \cdot 1_{\alpha} = d(\alpha)\xi$ for all $\alpha \in \textbf{Irr}(\mcal C)$. A net $\left(\xi_i\right)_i $ in $H$ is called \textit{almost invariant} if $\xi_i \cdot 1_{\alpha} -d(\alpha)\xi_i \to 0 $ for all $\alpha \in \textbf{Irr}(\mcal C)$. 
\end{defn}
Given a finite set $F \subset \textbf{Irr}(\mcal C)$ and $\epsilon > 0$, a vector $\xi \in H$ is called $\left(F,\epsilon\right)$-invariant if \[\norm{\xi \cdot 1_{\alpha} - d(\alpha)\xi} < d(\alpha) \epsilon \norm{\xi} \ \ \ \ \t{for all} \ \ \alpha \in F \]
There exists an almost invariant net of unit vectors in $H$ if and only if $H$ admits $\left(F,\epsilon\right)$-invariant vectors for all finite sets $F \subset \textbf{Irr}(\mcal C)$ and for all $\epsilon > 0$.

\begin{defn}
	A pair $\left(F, \epsilon\right)$, with $F$ a finite subset of $\textbf{Irr}(\mcal C)$ and $\epsilon > 0 $, is called a \textit{Kazhdan pair} if any non-degenerate right Hilbert $\mcal A$-module containing a $\left(F, \epsilon\right)$-invariant vector must contain a nonzero invariant vector. We say that $\mcal C$ has \textit{property (T)} if it has a Kazhdan pair.
\end{defn}
 Thus, $\mcal C$ has property (T) if and only if a non-degenerate right Hilbert $\mcal A$-module with almost invariant vectors admits a non-zero invariant vector.

\subsection{Semi-pre-C*-algebras}\

Suppose $A$ is a unital-$ * $-algebra and consider the set of its hermitian elements $A_h \coloneqq \left\{a \in A : a^* = a\right\}$. We call a subset $A_+ \subset A_h$ to be a \textit{$ * $-positive cone} if it satisfy the following :
\begin{enumerate}
	\item $\R_{\geq 0} 1 \subset A_+$ and $\lambda a + b \in A_+$ for all $a, b \in A_+$ and for all $\lambda \in \R_{\geq 0}$.
	\item $x^* a x \in A_+ $ for all $a \in A_+$ and for all $x \in A$.
\end{enumerate}

Given a $ * $-positive cone $A_+$ of $ A $, we define the $ * $-subalgebra of bounded elements by
\[A^{\t{bdd}} = \left\{x \in A : \exists R >0 \ \ \t{such that} \ \ x^* x \leq R 1\right\}\]

\begin{defn}\cite{Oz13}
	A unital $ * $-algebra $A$ is called a \textit{semi-pre-C*-algebra} if $A = A^{\t{bdd}}$
\end{defn}    
   
\section{Main theorem}
Suppose $\mcal C$ is a rigid C*-tensor category with simple unit $\mathbbm{1}$ and having at most countably many isomorphism classes of simple object.
Let $\textbf{Irr}(\mcal C)$, or simply $ \textbf{Irr} $, denote a set of representatives of the isomorphism classes of simples in $ \mcal C $.
Consider the fusion algebra $\C [\mcal C]$ of $\mcal C$. At times, we will use the graphical calculus for C*-tensor categories, we denote objects by strands and morphisms by boxes with strands passing through it. For a detailed study about graphical calculus of rigid C*-tensor categories, we refer the reader to \cite{HV}. For notational convenience, we will drop the $[\cdot]$-notation for elements of $\C[\mcal C]$, so for an element $[x]$ in $\C[\mcal C]$ we will write it as $x$ itself.

\vspace*{4mm}

We will identify $\C [\mcal C]$ as a $ * $-subalgebra of the tube algebra $ \mcal A \mcal C $ of $ \mcal C $,  and exhibit a $ * $-positive cone in it.
Recall from \cite{O94,GhJ}, $ \mcal A \mcal C $, or simply $  \mcal A$, is defined as $ \mcal A \coloneqq \displaystyle \bigoplus_{x,y \in \textbf{Irr}} \mcal A_{x,y}$ where $ \mcal A_{x,y} \coloneqq \displaystyle \bigoplus_{w \in \textbf{Irr}} \mcal A^w_{x,y}$ and $ \mcal A^w_{x,y} \coloneqq \mcal C \left( w \otimes  x  ,  y \otimes w \right) \; $ ($ \oplus $ denotes direct sum of vector spaces).
For $ x,y \in \textbf{Irr} $ and $ z \in \t{ob} \left(\mcal C\right) $, we have canonical maps
\[
\mcal C \left( z \otimes x , y \otimes z\right) \ni \gamma \os{\displaystyle \Psi^z_{x,y}}{\longmapsto} \left(\sum_{\alpha\, \in\, \t{onb}\,  \mcal C \left( w,z \right)} \left(1_y \otimes \alpha^*\right)\ \gamma \ \left(\alpha \otimes 1_x\right) \right)_{w \in \textbf{Irr}}
\]
for $ x,y \in \textbf{Irr} $ where $ \alpha  $ varies over any orthonormal basis in $ \mcal C \left( w , z\right) $ with inner product given by $ \tau^* \sigma = \left\lab \sigma , \tau \right \rab \ 1_{w} $.
By \cite{GhJ}, any $ a \in \mcal A_{x,y} $ can be expressed as $ \Psi^z_{x,y}(\gamma) $ for some $ z \in \t{ob} \left( \mcal C \right) $ and $ \gamma \in \mcal C \left( z \otimes x , y \otimes z\right) $.
The multiplication and the $ * $-structure is given by
\[
\mcal A_{x_1 , y_1} \times \mcal A_{x_2, y_2} \ni \left(\Psi^{z_1}_{x_1,y_1} \left(\gamma_1\right) , \Psi^{z_2}_{x_2,y_2} \left(\gamma_2\right)\right) \longmapsto \delta_{x_1 = y_2} \Psi^{z_1 \otimes z_2}_{x_2,y_1} \left(\left(\gamma_1 \otimes 1_{w_2}\right)\left(1_{w_1} \otimes \gamma_2\right)\right) \in \mcal A_{x_2, y_1}
\]
\[
\t{and }\; \mcal A_{x,y} \ni \Psi^{z}_{x,y} \left(\gamma\right) \longmapsto \Psi^{\ol z}_{y,x} \left(\left(R^*_z \otimes 1_x \right)\left(1_{\ol z} \otimes \gamma^* \otimes 1_{\ol z}\right) \left(1_y \otimes \ol R_z\right)\right) \in \mcal A_{y,x}
\]
respectively where $ \left(R_z , \ol R_z\right) $ is a balanced standard solution to the conjugate equation implementing the duality of $ \left(z , \ol z\right) $.
Indeed, the above maps are well defined; (see \cite{DGG14, GhJ}).
Further, we have a faithful positive definite functional $ \Omega : \mcal A \ra \C $ defined by
\[
\mcal A_{x,y} \supset \mcal A^w_{x,y} \ni \gamma \os {\displaystyle\Omega} \longmapsto \delta_{x=y} \, \delta_{w= \mathbbm{1}} \, \t{Tr}_x (\gamma) \in \C
\]
where $ \t{Tr}_x $ denotes the canonical unnormalized trace on $ \t{End} (x) $ induced by any balanced spherical solution to the conjugate equation implementing the duality of $ x $.
\begin{rem}\label{projrem}
For $m \in \textbf{Irr}$, there is a projection $p_m \in \mcal A^{\mathbbm{1}}_{m,m}$ given by $p_m \coloneqq 1_m \in \mcal C(\mathbbm{1} \otimes m, m \otimes \mathbbm{1})$. In particular, we have $\left(p_m\right)^k_{i,j} = \delta_{k,\mathbbm{1}} \delta_{i,j} \delta_{j,m} 1_m $. Clearly, $\mcal A_{x,y} = p_y \mcal A p_x $. For every $x \in \textbf{Irr}$, the corner algebras $\mcal A_{x,x}$ are unital $ * $-algebras (where  the unit is given by $p_x$) along with a faithful positive definite functional given by $ \left. \Omega \right|_{\mcal A_{x,x}} $.
\end{rem}
The $ * $-subalgebra $ \mcal A_{\mathbbm 1, \mathbbm 1} $ turns out to be isomorphic to the fusion algebra via the map (which henceforth, will act as an identification between the two spaces)
\[
\C [\mcal C] \supset \textbf{Irr} \ni w \longmapsto \Psi^w_{\mathbbm 1,\mathbbm 1 } \left(1_w\right) \in \mcal A_{\mathbbm 1,\mathbbm 1} \; .
\]
Note that the identity $ \mathbbm 1 $ of $ \C \left[\mcal C \right] $ gets mapped to $ \Psi^\mathbbm{1}_{\mathbbm{1},\mathbbm{1}} \left(1_\mathbbm{1}\right) \in \mcal A_{\mathbbm{1},\mathbbm{1}} $.
Henceforth, we will identify the two spaces $ \C \left[\mcal C\right] $ and $ \mcal A_{\mathbbm{1},\mathbbm{1}} $ using this map.
\begin{defn}
The \textit{tube cone of $ \mcal C $} is defined as
\[
\Sigma^2 \mcal C \coloneqq \left\{ \sum_{i=1}^{n} a_i a_i^* : a_i \in \us{x_i \in \textbf{Irr}}{\bigcup} \mcal A_{x_i,\mathbbm{1}}, 1 \leq i \leq n, n \in \N \right\} \; . 
\]
\end{defn}
\begin{lem}
	$\Sigma^2 \mcal C$ is a $ * $-positive cone of $\C [\mcal C]$.
\end{lem}  
\begin{proof}
The straight forward verification is left to the reader.
\end{proof}
We call an element $x \in \C [\mcal C]$ \textit{positive} if $x \in \Sigma^2 \mcal C$, and we denote it by $x \geq 0$.
\begin{prop}\label{salient}
$ \Sigma^2 \mcal C $ is a `salient' cone, that is, $ \Sigma^2 \mcal C\, \cap\, \left(- \Sigma^2 \mcal C\right) = \{0\}$.
\end{prop}
\begin{proof}
It is enough to show that for $ a_j \in \mcal A_{x_j, \mathbbm{1}} $, $ 1\leq j \leq n $, the equation $ \displaystyle \sum^n_{j=1} a_j a^*_j = 0 $ implies $ a_j = 0 $ for all $ j $.
Applying $ \Omega $ on both sides and using its faithfulness, we get the desired result.
\end{proof}
As a result, the salient cone $ \Sigma^2 \mcal C $ induces a unique partial order on $ \C \left[\mcal C \right]_\t{sa} $ (the space of self-adjoint elements  of $ \C \left[\mcal C \right] $).
A simple polarization identity argument will imply that every element of $ \C \left[ \mcal C\right]_\t{sa} $ is a difference of a pair in $ \Sigma^2 \mcal C $ (see \cite{S08}).
\begin{prop}\label{orderunit}
$ \mathbbm 1 $ turns out to be an `order unit' in the cone $ \Sigma^2 \mcal C $, that is, for every self-adjoint $ a \in \C \left[\mcal C \right]_\t{sa}$, there exists scalar $ R >0 $ such that $ a + R \mathbbm 1 \in \Sigma^2 \mcal C$.
\end{prop}
\begin{proof}
It is enough to prove that for all $ b = \Psi^z_{\mathbbm 1, x} (\gamma) \in \mcal A_{\mathbbm 1, x} $ where $ \gamma \in \mcal C \left( z, x \otimes z \right)$, there exists scalar $ R >0 $ such that $ R \mathbbm 1 -  \left(b^* \cdot b\right) \in \Sigma^2 \mcal C $.
Note that $ \left(b^* \cdot b\right) $ can be expressed as
$ \Psi^{z \otimes \ol z}_{\mathbbm 1 , \mathbbm 1} \left(
\raisebox{-1.15cm}{
	\begin{tikzpicture}
		\draw (-.2,0) to (-.2,-1.1);
		\draw[->] (.6,-.35) to[out=-90,in=-90] (0.2,-.35) to[out=90,in=-90] (0,.4) to[out=90,in=90] (-.6,.4) to (-.6,-1.8);
		\draw[<-] (.6,-.35) to (.6,.7);
		\draw[->] (0,-1.8) to (0,-1.5);
		\draw[->] (0,-1.5) to (0,-1.1) to[in=-90] (0.2,-.8) to[out=90,in=90] (.6,-.8) to (.6,-1.6) to[out=-90,in=-90] (1,-1.6) to (1,.7);
		\node[draw,thick,rounded corners, fill=white] at (0,0) {$ \gamma^* $};
		\node[draw,thick,rounded corners, fill=white,minimum width=18] at (0,-1.1) {$ \gamma $};
		\draw[dashed,red,thick] (-.45,-1.5) to (-.45,.45) to (.8,.45) to (.8,-1.5) to (-.45,-1.5);
\node at (-.05,-.6) {$ x $};
	\end{tikzpicture}
} \right) $
where we use balanced standard solution conugate equations implementing the duality of $ z $.
Let $ \sigma $ be the element in the finite dimensional C*-algebra $ \t{End} \left(z \otimes \ol z\right) $ appearing in the dotted box; note that $ \sigma $ is positive.
Set $ \tau \coloneqq \left[ \norm{\sigma} 1_{z \otimes \ol z} - \sigma \right]^{\frac 1 2}  \in \t{End} \left( z \otimes \ol z \right)$ and $ c_{y,\alpha} \coloneqq \Psi^z_{\mathbbm 1 , y} \left(
\raisebox{-.9 cm}{
	\begin{tikzpicture}
		\draw (.6,1.5) to (.6,.6);
		\draw (0,.9) to (0,1.5);
		\draw[->] (-.2,-.6) to (-.2,-.3);
		\draw[->] (-.2,-.3) to (-.2,.5);
		\draw (-.2,.5) to (-.2,.9);
		\draw[->] (.2,.6) to (.2,0.4);
		\draw[->] (.2,.4) to (.2,-.2) to[out=-90,in=-90] (.6,-.2) to (.6,.6);
		\node[draw,thick,rounded corners, fill=white,minimum width=20] at (0,0) {$ \tau $};
		\node[draw,thick,rounded corners, fill=white,minimum width=20] at (0,.9) {$ \alpha^* $};
\node at (-.15,1.4) {$ y $};
	\end{tikzpicture}
} \right) \in\mcal A_{\mathbbm 1, y}$ where $ y \in \textbf{Irr} $ and $ \alpha \in \t{onb}\, \mcal C \left( y, z \otimes \ol z \right)$.
Observe that only finitely many $ c_{y,\alpha} $'s are nonzero.
Now, 
\[
\Sigma^2 \mcal C \ni  \sum_{y\in \textbf{Irr}} \; \sum_{\t{onb} \mcal C (y,z \otimes \ol z)} c^*_{y,\alpha} c_{y,\alpha} = \Psi^{\ol z \otimes z}_{\mathbbm 1 ,\mathbbm 1 } \left(
\raisebox{-.5cm}{
	\begin{tikzpicture}
		\draw[->] (-.2,-.6) to (-.2,.3) to[out=90,in=90] (-.6,.3) to (-.6,-.6);
		\draw[->] (.2,.6) to (.2,-.3) to[out=-90,in=-90] (.6,-.3) to (.6,.6);
		\node[draw,thick,rounded corners, fill=white,minimum width=20] at (0,0) {$ \tau^2 $};
	\end{tikzpicture}
} \right) = \norm{\sigma} \, d(z) \, \mathbbm 1 \; - \; b^* b .
\]
\end{proof}
\begin{cor}\label{alphacor}
$z^* \cdot z \; \leq \; {d(z)}^2 \mathbbm{1}$ \; for each $z \in \normalfont \textbf{Irr}$.
\end{cor}
\begin{proof}
Setting $ x=\mathbbm 1 $ and $ \gamma = 1_z $ in the proof of \Cref{orderunit}, $ b $ gets identified with $ z $.
Note that $ \left[d(z)\right]^{-1} \sigma $ is a projection; thus $ \norm{\sigma} = d(z) $.
This ends the proof.
\end{proof}
\begin{thm}
	$\C [\mcal C]$ is a semi-pre-C*-algebra with respect to the cone $\Sigma^2 \mcal C$.
\end{thm}
\begin{proof}
Pick a generic element $ b = \Psi^z_{\mathbbm 1,\mathbbm 1}  (\gamma) \in \mcal A_{\mathbbm 1,\mathbbm 1}$ where $ z \in \t{ob} (\mcal C) $ and  $ \gamma \in \t{End} (z) $.
We need to show $ b \in {\C [ \mcal C]}^{\t{bdd}} $, that is, $ b^* \cdot b \leq R \, \mathbbm 1 $ with respect to the tube cone
for some $ R \in [0,\infty) $.
This easily follows from the proof of \Cref{orderunit} after substituting $ x = \mathbbm 1 $.
\end{proof}
\begin{defn}
	A subset $S$ of $\textbf{Irr}$ is said to be
	\begin{itemize}
		\item [(i)] \textit{symmetric}, if, $\alpha \in S$ if and only if $\ol \alpha \in S$.
		\item [(ii)] \textit{generating}, if, for any $\alpha \in \textbf{Irr}$, $\alpha$ is a sub-object of $s_1 \otimes s_2 \otimes \cdots \otimes s_n$ for some $s_1, s_2, \cdots, s_n  \in S $. 
	\end{itemize}
\end{defn}
Suppose $\mcal C$ has a finite, symmetric, generating set $ S $. Let $\nu : S \to \left(0, \infty\right)$ be a symmetric (that is, $\nu(\alpha) = \nu(\ol \alpha)$ for all $\alpha \in S$) weight function. Given such a triplet $\left(\mcal C, S, \nu \right)$ one can form a \textit{Laplacian} element $\Delta$ in the fusion algebra $\C [\mcal C]$ given by 
\[ \Delta \coloneqq \mathbbm{1} - \frac{1}{\kappa} \us{\alpha \in S}{\sum} \nu(\alpha) \alpha, \ \ \ \ \t{where} \ \ \kappa = \us{\alpha \in S}{\sum} \nu(\alpha) d(\alpha)   \]
Our goal is to prove a version of Ozawa's criterion for property (T) for a C*-tensor category $\mcal C$ in terms of the Laplacian element $\Delta$. 
Throughout our discussion, the triplet $\left(\mcal C, S, \nu \right)$ will remain fixed.
\begin{rem}
	In the group case, positivity of $\Delta$ was inherent because the basis elements of the group algebra were invertible. We do not have this advantage in the case of the fusion algebra $\C [\mcal C]$.
\end{rem}
In the next proposition, we prove the positivity of $\Delta$ with respect to the cone $\Sigma^2 \mcal C$.
\begin{prop}
	$\Delta \geq 0$ in $\C [\mcal C]$.
\end{prop}
\begin{proof}
	Observe that we have, for all $\alpha \in S$, $$\left(1-\frac{\alpha}{d(\alpha)}\right)^* \cdot \left( 1- \frac{\alpha}{d(\alpha)}\right) \geq 0 \ .$$
Also,
	 \begin{align*}
	 	\left(1-\frac{\alpha}{d(\alpha)}\right)^* \cdot \left( 1- \frac{\alpha}{d(\alpha)}\right) &= \left(1-\frac{\ol \alpha}{d(\alpha)}\right) \cdot \left( 1- \frac{\alpha}{d(\alpha)}\right) \\ &=1 + \frac{\ol{\alpha} \cdot \alpha}{d(\alpha)^2} - \frac{1}{d(\alpha)} \left(\alpha + \ol \alpha \right) \\ &\leq 2 - \frac{1}{d(\alpha)}\left(\alpha + \ol \alpha \right) \ \ \left(\t{By} \ \ \Cref{alphacor} \right).  
	 \end{align*}
 Thus, $\alpha + \ol \alpha \leq 2 d(\alpha)$ for all $\alpha \in S$. Therefore, $\us{\alpha \in S}{\sum} \nu (\alpha) \left(\alpha + \ol \alpha \right) \leq 2 \kappa $. Since, $\nu$ is symmetric, we have that $\Delta \geq 0$ .
\end{proof}

\vspace*{2mm}

\begin{defn}\cite{GhJ}
	Consider the projection $p_{\mathbbm{1}}$ corresponding to the unit object $\mathbbm{1}$ as in \Cref{projrem}. A linear functional $\phi : \mcal A_{\mathbbm{1},\mathbbm{1}} \to \C$ is called a \textit{weight $\mathbbm{1}$-annular state} if :
	\begin{itemize}
		\item [(i)] $\phi \left(p_{\mathbbm{1}}\right) = 1 $.
		\item [(ii)] $\phi \left(f^* \cdot f \right) \geq 0$ for all $f \in \mcal A_{\mathbbm{1},m}$ and $m \in \textbf{Irr}$.
	\end{itemize}
\end{defn}

\begin{defn}
	A non-degenerate $ * $-representation $\left(\pi, H \right)$ of $\C [\mcal C]$ is said to be a \textit{weight $\mathbbm{1}$-admissible representation}, if every vector state in $\left(\pi, H \right)$ is a weight $\mathbbm{1}$-annular state.
\end{defn}

We now state a result connecting property (T), $\Delta$ and weight $\mathbbm{1}$-admissible representations.

\begin{lem}\label{admissiblelem}
	$\mcal C$ has property (T) if and only if $ 0 $ is not a limit point of the spectrum of $\pi(\Delta)$ for every weight $\mathbbm{1}$-admissible representation $\left(\pi, H \right)$ of $\C [\mcal C]$. 
\end{lem}
\begin{proof}
	This is just a reformulation of \cite[Lemma 3.3]{VV19}.
\end{proof}

\comments{\red{To add one or two lines about the importance of admissible representations}\\}
Admissible representations can be seen simply as representations of the centralizer algebras which are
restrictions of representations of the whole tube algebra. Alternatively, they are representations of the corner
algebras which induce representations of the whole tube algebra. Understanding admissible representations
for all weights allows us to understand representations of the whole tube algebra.

In the next proposition we explore the relation between weight $\mathbbm{1}$-admissible representations and $\Sigma^2 \mcal C$-positive representations of $\C [\mcal C]$ on a Hilbert space $H$ (that is, those unital $ * $-representations of $\C[\mcal C]$, that map elements of $\Sigma^2 \mcal C $ to positive elements of $B(H)$).

\begin{prop}\label{admissiblepositive}
	Suppose $\left(\pi, H\right)$ is a non-degenerate $ * $-representation of $\C [\mcal C]$. The following are equivalent:
	\begin{itemize}
		\item [(i)] $\pi$ is weight $\mathbbm{1}$-admissible.
		\item [(ii)] $\pi$ is $\Sigma^2 \mcal C$-positive.
	\end{itemize}
\end{prop} 
\begin{proof}	
	$\pi$ is weight $\mathbbm{1}$-admissible if and only if every vector state $\left(\pi, H \right)$ is a weight $\mathbbm{1}$-annular state. This is equivalent to,
	 $$\lab \pi\left(f^* \circ f\right)\xi, \xi \rab \geq 0, \ \text{for all} \ \xi \in H, f \in \mcal A_{\mathbbm{1},x} \ \text{and} \ x \in \textbf{Irr}.$$
	 This in turn is equivalent to $\pi(y) \geq 0$, in $B(H)$ for all $y \in \Sigma^2 \mcal C$.
\end{proof}

Using \Cref{admissiblelem} and \Cref{admissiblepositive}, we prove the following result.

\begin{thm}
	Suppose $\mcal C$ is a rigid, C*-tensor category having a finite, symmetric, generating set $S$ and a symmetric weight function $\nu : S \to \left(0, \infty\right)$. $\mcal C$ has property (T) if and only if for some $k > 0$, $\Delta^2 - k \Delta + \epsilon \mathbbm{1} \in \Sigma^2 \mcal C $ for all $\epsilon > 0$.
\end{thm}
\begin{proof}
By \Cref{admissiblelem}, we have that $\mcal C$ has property (T) if and only if $ 0 $ is an isolated point of the spectrum of $\pi(\Delta)$ for every weight $\mathbbm{1}$-admissible representation $\left(\pi, H \right)$ of $\C [\mcal C]$. An easy application of \Cref{admissiblepositive} reveals that $\mcal C$ has property (T) if and only if $ 0 $ is an isolated point of the spectrum of $\pi(\Delta)$ for every $\Sigma^2 \mcal C$-positive representation $\left(\pi, H \right)$ of $\C [\mcal C]$. Thus, we have that, $\mcal C$ has property (T) if and only if for some $k > 0$, $\sigma \left(\pi(\Delta)\right) \subseteq \left\{0\right\} \cup [\,k, \infty )\,$  for every $\Sigma^2 \mcal C$-positive representation $\left(\pi, H \right)$ of $\C [\mcal C]$. This is equivalent to $\pi (\Delta^2 -k \Delta) \geq 0$ in $B(H)$ for every $\Sigma^2 \mcal C$-positive representation $\pi$ of   $\C [\mcal C]$ 
  which in turn is equivalent to  $\Delta^2 -k \Delta + \epsilon \mathbbm{1} \in \Sigma^2 \mcal C$, for all $\epsilon > 0$ (Using \cite[Proposition 14]{S08}).
  This concludes the theorem	
	
\end{proof}

\Contact


\end{document}